\documentclass[12pt,reqno,sort]{elsarticle}

\usepackage{amsfonts,amsmath,amssymb,fancyhdr}
\usepackage{amsthm}
\usepackage{boondox-cal}
\usepackage{mathrsfs}
\usepackage{subfigure}
\usepackage{booktabs}
\usepackage{multirow}
\usepackage{mathrsfs}
\usepackage{booktabs}
\usepackage{longtable}
\usepackage[cal=cm]{mathalfa}
\def\Box{\vcenter{\vbox{\hrule\hbox{\vrule
				\vbox to 8.8pt{\hbox to 10pt{}\vfill}\vrule}\hrule}}}

\newtheorem{thm}{Theorem}[section]
\newtheorem{lemma}[thm]{Lemma}
\newtheorem{corollary}[thm]{Corollary}

\newtheorem{example}[thm]{Example}
\numberwithin{equation}{section}
\newtheorem{remark}[thm]{Remark}
\newtheorem{hypothesis}[thm]{Hypothesis}

\def\a{\alpha}

\begin{document}
	\newcommand{\stopthm}{\begin{flushright}
			\(\box \;\;\;\;\;\;\;\;\;\; \)
	\end{flushright}}
	\newcommand{\symfont}{\fam \mathfam}
	
	\title{Flag-transitive point-primitive quasi-symmetric $2$-designs and exceptional groups of Lie type}
	
	\date{}
	
	\author[add1]{Jianbing Lu}\ead{jianbinglu@nudt.edu.cn}
	\address[add1]{Department of Mathematics, National University of Defense Technology, Changsha 410073, China}

	\begin{abstract}
		Let $\mathcal{D}$ be a non-trivial quasi-symmetric $2$-design with two block intersection numbers $x=0$ and $2\leq y\leq10$, and suppose that $G$ is an automorphism group of $\mathcal{D}$. If $G$  is flag-transitive and point-primitive, then it is known that $G$ is either of  affine type or almost simple type. In this paper, we show that the socle of $G$ cannot be a finite simple exceptional group of Lie type. 
		\newline
		
		\noindent\text{Keywords:} quasi-symmetric $2$-design; flag-transitive; point-primitive; automorphism group; exceptional group of Lie type
		
	\end{abstract}	
	
	\maketitle
	
	\section{Introduction}\label{introduction}
	
	A $2$-$(v,k,\lambda)$ design $\mathcal{D}=(\mathcal{P},\mathcal{B})$ is  a finite incidence structure  with a set $\mathcal{P}$ of $v$ points and a set $\mathcal{B}$ of blocks such that each block contains $k$ points and each two distinct points are contained in  $\lambda$ blocks. It is \emph{non-trivial} if $2<k<v-1$. All the $2$-designs in this paper are assumed to be non-trivial. The replication number $r$ of $\mathcal{D}$ is the number of blocks containing a given point. The number of blocks is conventionally denoted by $b$. If $b=v$, we say that $\mathcal{D}$ is \emph{symmetric}.  Let $\mathcal{D}$ be a $2$-$(v,k,\lambda)$ design with  blocks $B_1,B_2,\cdots,B_b$. The cardinality $|B_i\cap B_j|$ is called a block intersection number of $\mathcal{D}$.  If $D$ is symmetric, then  $D$  has only one block intersection number, namely $\lambda$. Those $2$-designs with two block intersection numbers are called \emph{quasi-symmetric}. This concept  goes back to \cite{Shrikhande1952}. Let $x$, $y$ denote the two block intersection numbers of a quasi-symmetric design  with the standard convention that $x\leq y$. We say that quasi-symmetric design is proper if $x<y$ and improper if $x=y$. There are many well known examples of proper quasi-symmetric $2$-designs. For example, a linear space with $b>v$ is a quasi-symmetric design with $x=0$ and $y=1$. We refer to \cite{Shrikhande, Neumaier1982} for more details and additional examples of quasi-symmetric designs. 
	
	An \emph{automorphism} of $\mathcal{D}$ is a permutation of the points which preserves the blocks. We write $\mathrm{Aut}(\mathcal{D})$ for the full automorphism group of $\mathcal{D}$, and call its subgroups as automorphism groups. A \emph{flag} of  $\mathcal{D}$ is an incident point-block pair. We say that an automorphism group $G$ is \emph{flag-transitive} if it acts transitively on the flags of $\mathcal{D}$. If $G$ acts primitively on the points of $\mathcal{D}$, then $G$ is said to be \emph{point-primitive}.   There have been extensive works on the classification of flag-transitive point-primitive $2$-designs. In 1990, Buekenhout, Delandtsheer, Doyen, Kleidman, Liebeck and Saxl \cite{Buekenhout1990} classified all flag-transitive linear spaces apart from those with a one-dimensional affine automorphism group. Through a series of papers \cite{O'Reilly2005,O'Reilly2005_2,O'Reilly2007,O'Reilly2008},  Regueiro gave the classification  of  biplanes ($\lambda=2$). However, moving to the non-symmetric case poses additional challenges, see \cite{Alavi2023, Alavi2024, Deviller, Liang2016, Liang2016_2, Liang2024}. More recent and interesting classification results are provided in \cite{Li2024, Montinaro2023, Montinaro2024}. 
	
	This paper is a contribution to the study of  quasi-symmetric $2$-designs admitting
	a flag-transitive point-primitive automorphism group. In \cite{Lu}, we proved that for a non-trivial quasi-symmetric $2$-design $\mathcal{D}$ with two block intersection numbers $x=0$ and $2\leq y\leq10$, if $G\leq \mathrm{Aut}(\mathcal{D})$ is flag-transitive and point-primitive, then $G$ is either of  affine type or almost simple type. Moreover,  we proved that the socle of $G$ cannot be an  alternating group, and  gave the classification results for the case where the socle of $G$ is a sporadic simple group. In this paper, we further investigate the  case where the
	socle of $G$ is a finite simple exceptional group of Lie type. The following are our main results.

	\begin{thm}\label{main}
		Let $\mathcal{D}$ be a non-trivial quasi-symmetric $2$-design with block intersection numbers $x=0$ and $2\leq y\leq 10$. If $G\leq\mathrm{Aut}(\mathcal{D})$ is flag-transitive and point-primitive, then the socle of $G$ cannot be a finite simple exceptional group of Lie	type.
	\end{thm}
	
	Combining Theorem \ref{main} with \cite[Theorem 1.2, 1.4]{Lu}, one obtains the following corollary.
	
	\begin{corollary}\label{second1}
		Let $\mathcal{D}$ be a non-trivial quasi-symmetric $2$-design with block intersection numbers $x=0$ and $2\leq y\leq 10$. Let $G\leq\mathrm{Aut}(\mathcal{D})$ be flag-transitive and point-primitive. If the socle of $G$ is not a finite simple classical group, then $\mathcal{D}$ and $G$ are one of the following:
		\begin{enumerate}
			\item[(1)] $\mathcal{D}$ is the unique $2$-$(12,6,5)$ design with block intersection numbers $0$ and $3$, $G=\mathrm{M}_{11}$;
			\item[(2)] $\mathcal{D}$ is the unique $2$-$(22,6,5)$ design with block intersection numbers $0$ and $2$, $G=\mathrm{M}_{22}$ or $\mathrm{M}_{22}:2$.
		\end{enumerate}
	\end{corollary}
	
	In \cite{Saxl2002}, Saxl completed the classification of  the  linear spaces admitting almost simple flag-transitive automorphism groups. Combining Theorem \ref{main} with the result in \cite{Saxl2002}, one obtains the following corollary.
	
	\begin{corollary}\label{second2}
		Let $\mathcal{D}$ be a non-trivial quasi-symmetric $2$-design with block intersection numbers $x=0$ and $y\leq 10$. If $G\leq\mathrm{Aut}(\mathcal{D})$ is flag-transitive and point-primitive with exceptional socle of Lie type, then  $\mathcal{D}$ is a Ree unital and  $\mathrm{Soc}(G)={ }^2G_2(q)$ with $q=3^{2n+1}$.
	\end{corollary}

	This paper is organized as follows. In Section \ref{s2}, we present some preliminary results on flag-transitive $2$-designs and the maximal subgroups of almost simple groups with exceptional socle of Lie type. Then the proof of  Theorem \ref{main} is presented in Section \ref{s3}.
	
	\section{Preliminaries}\label{s2}
	
	\subsection{Flag-transitive $2$-designs}
	We first present some preliminary results on flag-transitive $2$-designs which are used throughout the paper.
	
	\begin{lemma}\label{s21}
		\cite[1.2 and 1.9]{Colburn} Let $\mathcal{D}$ be a $2$-design with the parameters $(v,b,r,k,\lambda)$. Then the following hold:
		\begin{enumerate}
			\item[(i)] $r(k-1)=\lambda(v-1)$, $r>\lambda$;
			\item[(ii)] $vr=bk$;
			\item[(iii)] $b\geq v$, $k\leq r$. If $\mathcal{D}$ is non-symmetric, then $b>v$ and $k<r$.
		\end{enumerate}
	\end{lemma}
	
	\begin{lemma}\label{s23}
		\cite[Lemma 2.4]{Zhang2023}	Let $\mathcal{D}$ be a $2$-design with the parameters $(v,b,r,k,\lambda)$, $\alpha$ be a point of $\mathcal{D}$ and $G$ be a flag-transitive automorphism group of $\mathcal{D}$. Then the following hold:
		\begin{enumerate}
			\item[(i)] 
			$|G_{\alpha}|^{3}>\lambda|G|$;
			\item[(ii)] $r\mid \lambda(v-1,|G_{\alpha}|)$, where $G_{\alpha}$ is the stabilizer of $\alpha$ in $G$;
			\item[(iii)] If $d$ is any non-trivial subdegree of $G$, then $r\mid\lambda d$. In particular, $\frac{r}{(r,\lambda)}\mid (v-1,d)$.
		\end{enumerate}
	\end{lemma}
	
	It is well known that $\mathcal{D}$ is a quasi-symmetric design with intersection numbers $0$ and $1$ if  and only if $\mathcal{D}$ is a finite linear space with  $b>v$ (see \cite[Chapter 3]{Shrikhande}). For the linear space  $\mathcal{D}$ with a flag-transitive automorphism group $G$, we know that $G$ is point-primitive \cite{Higman1969}. Moreover, $G$ is either almost simple or of affine type \cite{Buekenhout1988}. If   the socle of $G$ is a finite simple
	exceptional group of Lie type, then $\mathcal{D}$ is a Ree unital and $\mathrm{Soc}(G)={ }^2G_2(q)$ with $q=3^{2n+1}$ (see \cite{Delandtsheer1986,Kleidman1990,Saxl2002}).  Thus we suppose that $\mathcal{D}$ is a quasi-symmetric $2$-design with block intersection numbers $x=0$ and $y\geq2$ in the following. Then we have the following lemma.

	\begin{lemma}\label{s22}
		Let $\mathcal{D}$ be a non-trivial quasi-symmetric $2$-design with block intersection numbers $x=0$ and $y\geq2$. Then the following relations hold:
		\begin{enumerate}
			\item[(i)] $(y-1)(r-1)=(k-1)(\lambda-1)$;
			\item[(ii)] $b\leq v(v-1)/k$, $r\leq v-1$, $y< \lambda\leq k-1$;
			\item[(iii)] $y\mid k$, $y\mid(r-\lambda)$;
			\item[(iv)] $v<\frac{k^{2}-k}{y-1}$;
			\item[(v)] $	(y-1)\cdot\frac{r^2}{\lambda^2}<v-1<2(y-1)\cdot\frac{r^2}{\lambda^2}$. In particular,  $v\leq2(y-1)\cdot\frac{r^{2}}{(r,\lambda)^{2}}$. 
		\end{enumerate}
	\end{lemma}
	
	\begin{proof}
		It suffices to  prove part (v) since other  properties are taken from \cite[Lemma 2.3]{Lu}.  Note that
		\[(y-1)\cdot\frac{r}{\lambda}<k-1=(y-1)\cdot\frac{r-1}{\lambda-1}<2(y-1)\cdot\frac{r}{\lambda}\]
		and $k-1=(v-1)\cdot\frac{\lambda}{r}$ by Lemma \ref{s21}(i) and Lemma \ref{s22}(i). Then we obtain 
		\[(y-1)\cdot\frac{r}{\lambda}<k-1=(v-1)\cdot\frac{\lambda}{r}<2(y-1)\cdot\frac{r}{\lambda},\]
 proving part (v).
	\end{proof}

	\subsection{Maximal subgroups of  almost simple groups with exceptional socle of Lie type}
	To prove Theorem \ref{main}, we require some results on maximal subgroups of  almost simple groups with exceptional socle of Lie type. The \emph{socle} of a finite group is defined as the subgroup generated by   all the minimal normal subgroups. A group $G$ is said to be \emph{almost simple} with socle $X$ if $X\unlhd G\leq \mathrm{Aut}(X)$, where $X$ is a non-abelian simple group.  Throughout this paper, we denote by $[n]$ a group of order $n$. We also adopt the standard Lie notation for groups of Lie type.  Specifically, we write $A_{n-1}(q)$ and $A^{-}_{n-1}(q)$ in place of $\mathrm{PSL}(n,q)$ and $\mathrm{PSU}(n,q)$ respectively, and $B_{n}(q)$, $C_{n}(q)$, $D^{\pm}_{n}(q)$, and $E^{-}_6(q)$  instead of $\mathrm{P}\Omega(2n+1,q)$,  $\mathrm{PSp}(2n,q)$, $\mathrm{P}\Omega^{\pm}(2n,q)$,  and  ${}^2E_6(q)$ respectively. Further notation in group theory can be found in \cite{Kleidman, Dixon}. For a positive integer $n$ and prime number $p$, let $n_{p}$  denote the $p$-part of $n$, that is, $n_{p}=p^{t}$ where $p^{t}\mid n$ but $p^{t+1}\nmid n$. The first lemma is an elementary result on subgroups of almost simple groups.

	\begin{lemma}\label{stabilizer}
		\cite[Lemma 3.1]{Alavi2022} Let $G$ be an almost simple group with socle $X$, and let $H$ be maximal in $G$ not containing $X$. Then $G=HX$, and $|H|$ divides $|\mathrm{Out}(X)|\cdot|H\cap X|$.
	\end{lemma}
	
	A proper subgroup $H$ of $G$ is said to be large if $|H|^3\geq|G|$. In \cite{Alavi}, Alavi and Burness determined  all the large maximal subgroups of finite simple groups. Furthermore, in \cite{Alavi2022}, Alavi, Bayat and Daneshkhah determined all the large maximal subgroups of almost simple exceptional groups of Lie type.
	\begin{lemma}\label{max}
		\cite[Theorem 1.2]{Alavi2022} Let $G$ be a finite almost simple group whose socle $X$ is a finite simple exceptional group of Lie type, and let $H$ be a maximal subgroup of $G$ not containing $X$. If $H$ is a large subgroup of $G$, then $H$ is either parabolic, or one of the subgroups listed in Table \ref{large_maximal}.
	\end{lemma}
	\begin{table}
		\begin{center}
			\caption{Large maximal non-parabolic subgroups $H$ of   $G$ with socle $X$  in Lemma \ref{max}.}\label{large_maximal}
			\resizebox{\textwidth}{!}{
	\begin{tabular}{ccc}
		\hline$X$ & $H \cap X$ or type of $H$ & Conditions \\
		\hline
		${ }^2B_2(q)\left(q=2^{2n+1} \geq 8\right)$ & $(q+\sqrt{2 q}+1): 4$ &	$q=8,32$ \\
		        &   ${ }^2B_2\left(q^{1/3}\right)$	& $q>8,3 \mid 2 n+1$ \\
	    ${ }^2G_2(q)\left(q=3^{2 n+1} \geq 27\right)$ & 	$A_1(q)$ & \\
				&	${ }^2G_2\left(q^{1/3}\right)$ & 	$3 \mid 2 n+1$\\
	    ${ }^3D_4(q)$ & $A_1\left(q^3\right)A_1(q),\left(q^2+\epsilon q+1\right)A_2^\epsilon(q), G_2(q)$ & 	$\epsilon= \pm$\\
		    	&   ${ }^3D_4\left(q^{1/2}\right)$ & 	$q$ square\\
				&   $7^2: \mathrm{SL}(2,3)$ &	$q=2$ \\
        ${ }^2F_4(q)\left(q=2^{2 n+1} \geq 8\right)$ & ${ }^2B_2(q)\wr2, B_2(q):2,{ }^2F_4\left(q^{1/3}\right)$ & \\
				&   $\mathrm{SU}(3,q): 2, \mathrm{PGU}(3,q):2$ &	$q=8$ \\
				&    $A_2(3):2, A_1(25)$, $\mathrm{A}_6.2^2$, $5^2: 4\mathrm{A}_4$ &	$q=2$\\
        $G_2(q)$ & 	$A_2^\pm(q), A_1(q)^2, G_2\left(q^{1/r}\right)$ &	$r=2,3$ \\
				&   ${ }^2G_2(q)$ &	$q=3^a, a$ is odd \\
				&    $G_2(2)$ &	$q=5,7$ \\
				&    $A_1(13), \mathrm{J}_2$ &	$q=4$ \\
				&    $\mathrm{J}_1$ & $q=11$ \\
				&    $2^3.A_2(2)$ &		$q=3,5$ \\
        $F_4(q)$ & $B_4(q), D_4(q),{ }^3D_4(q)$ & \\
				&    $F_4\left(q^{1/r}\right)$ & $r=2,3$ \\
				&    $A_1(q)C_3(q)$ & 	$p \neq 2$ \\
				&    $C_4(q), C_2\left(q^2\right), C_2(q)^2$ & 	$p=2$ \\
				&    ${ }^2F_4(q)$ &	$q=2^{2n+1} \geq2$ \\
				&    ${ }^3D_4(2)$ &	$q=3$ \\
				&    $\mathrm{A}_{9}, \mathrm{A}_{10}, A_3(3), \mathrm{J}_2, \mathrm{S}_6\wr\mathrm{S}_2$ & 	$q=2$ \\
				&    $A_1(q)G_2(q)$ & 	$q>3$ odd \\
		$E_6^\epsilon(q)$ & $A_1(q)A_5^\epsilon(q), F_4(q)$ &  \\
				&    $(q-\epsilon) D_5^\epsilon(q)$ & 	$\epsilon=-$ \\
				&    $C_4(q)$ &	$p \neq 2$ \\
				&    $E_6^{\pm}\left(q^{1/2}\right)$ & 	$\epsilon=+$ \\
				&    $E^\epsilon_6\left(q^{1/3}\right)$ & \\
				&    $(q-\epsilon)^2.D_4(q)$ &	$(\epsilon, q)\neq(+, 2)$ \\
				&    $\left(q^2+\epsilon q+1\right).{ }^3D_4(q)$ & 	$(\epsilon, q)\neq(-,2)$ \\
                &	 $\mathrm{J}_3, \mathrm{A}_{12}, B_3(3), \mathrm{Fi}_{22}$ &	$(\epsilon, q)=(-,2)$ \\
        $E_7(q)$ & $(q-\epsilon)E_6^\epsilon(q), A_1(q)D_6(q), A_7^\epsilon(q), A_1(q)F_4(q), E_7\left(q^{1/r}\right)$ &	$\epsilon=\pm$ and $r=2,3$ \\
				&    $\mathrm{Fi}_{22}$ & 	$q=2$ \\
	    $E_8(q)$ & $A_1(q)E_7(q), D_8(q), A_2^\epsilon(q)E_6^\epsilon(q), E_8\left(q^{1/r}\right)$ & $\epsilon=\pm$ and $r=2,3$ \\
		\hline
	\end{tabular}}
\end{center}
\end{table}
	\begin{remark}
		\begin{enumerate}
			\item[(i)] For the orders of simple groups of Lie type, one can refer to, for example, \cite[Table 5.1.A, B]{Kleidman}.
			\item[(ii)] In Table \ref{large_maximal}, the type of $H$ is an approximate description of the group-theoretic structure of $H$. For precise structure, one may refer to \cite[Table 8.16]{Bray} for ${}^2B_2(q)$, \cite[Table 8.43]{Bray} for ${}^2G_2(q)$, \cite[Table 8.51]{Bray} for ${}^3D_4(q)$, \cite{Malle1991} for ${}^2F_4(q)$, \cite[Table 4.1]{Wilson2009} for $G_2(q)$, \cite{Craven2023} for $F_4(q)$ and $E^{\epsilon}_6(q)$, \cite{Craven2022} for $E_7(q)$ and \cite[Table 5.1]{Liebeck1992} for $E_8(q)$.
			\end{enumerate}
	\end{remark}

	\begin{lemma}\label{p-power}
		\cite[3.9]{Liebeck1987} If $X$ is a group of Lie type in characteristic $p$, acting on the set of cosets of a maximal parabolic subgroup, and $X$ is not $A_{n-1}(q)$, $D_{n}(q)$ (with $n$ odd) and $E_6(q)$, then there is a unique subdegree which is a power of $p$.
	\end{lemma}

	\section{Proof of Theorem \ref{main}}\label{s3}

 Throughout this section, we assume the following hypothesis.
		\begin{hypothesis}\label{hy_exceptional}
		Let $\mathcal{D}$ be a non-trivial quasi-symmetric $2$-design with block intersection numbers $x=0$ and $2\leq y\leq10$, let $G \leq \mathrm{Aut}(\mathcal{D})$ be flag-transitive  and point-primitive with socle $X$ a finite simple exceptional group of Lie type.  Let $\a$ be a point of $\mathcal{D}$ and $H=G_{\a}$.
	\end{hypothesis}

	Before proving the main theorem, we  introduce the following Computational methods through two examples.
	
	\subsection{Computational methods}\label{Comp}
	
	Our methods are based on   Lemma \ref{s22}(v) and Lemma \ref{s23}(ii)(iii). For  most of the large maximal non-parabolic subgroups listed in Table \ref{large_maximal}, we can use Lemma \ref{s23}(ii) to obtain  upper bounds on $\frac{r}{(r,\lambda)}$. Similarly, for parabolic subgroups, we  use Lemma  \ref{s23}(iii).  Combining these upper bounds with the lower bounds on $\frac{r}{(r,\lambda)}$ derived from Lemma \ref{s22}(v), we can deduce that  only a few small values of $q$ satisfy the conditions.
	\begin{example}\label{ex1}
		Suppose that $X=F_4(q)$ and $H\cap X=X_{\alpha}$ is a maximal subgroup of type $ {}^3D_4(q)$, where $q=p^f$ for some prime $p$ and positive integer $f$. According to \cite[Table 7]{Craven2023}, we know that $X_{\alpha}={}^3D_4(q).3$. Note that $|X|=q^{24}(q^2-1)(q^6-1)(q^8-1)(q^{12}-1)$ and $ |X_{\alpha}|=3q^{12}(q^8+q^4+1)(q^6-1)(q^{2}-1)$ (see \cite[Table 5.1.B]{Kleidman}). Then $v=|X|/|X_{\alpha}|=\frac{1}{3}q^{12}(q^8-1)(q^4-1)$. Next, we consider $3(v-1)$ and $|X_{\alpha}|$ as polynomials in $q$ over the rational field. Using the Magma \cite{Bosma} command $\mathbf{XGCD}$, we   obtain three rational polynomials $f(q)$, $g(q)$ and $h(q)$ such that
		\[
		f(q)\cdot 3(v-1)+g(q)\cdot|X_{\alpha}|=h(q).
		\]
		For this example, computation in Magma \cite{Bosma} shows that
		\[
		h(q)=q^8+q^4+1,
		\]
		\[
		f(q)=-\frac{5}{183}q^{18}+\frac{1}{183}q^{16}+\frac{13}{183}q^{14}-\frac{3}{61}q^{12}-\frac{1}{3}q^8-\frac{1}{3}q^4-\frac{1}{3},
		\]
		\[
		g(q)=\frac{5}{549}q^{14}+\frac{4}{549}q^{12}-\frac{19}{549}q^{10}-\frac{1}{183}q^{8}+\frac{22}{549}q^6+\frac{14}{183}q^4+\frac{73}{549}q^2+\frac{34}{549}.
		\]
		Notice that the least common multiple of the denominators of all coefficients of $f(q)$ and $g(q)$ is $549$. Then we deduce that $(3(v-1),|X_{\alpha}|)\leq 549(q^8+q^4+1)$. Since $|\mathrm{Out}(X)|=(2,p)f$, we have $\frac{r}{(r,\lambda)}\leq \left(v-1,|G_{\alpha}|\right)\leq\left(3(v-1),|X_{\alpha}|\cdot|\mathrm{Out}(X)|\right)\leq1098f\cdot (q^8+q^4+1)$ by Lemma \ref{s23}(ii) and Lemma \ref{stabilizer}.	Furthermore, from Lemma \ref{s22}(v), we deduce that $v\leq2(y-1)\cdot\frac{r^2}{(r,\lambda)^2}\leq 18\cdot \left(1098f\cdot (q^8+q^4+1)\right)^2$, i.e.,
		\begin{equation}
			\frac{1}{3}q^{12}(q^8-1)(q^4-1)\leq18\cdot \left(1098f\cdot (q^8+q^4+1)\right)^2. \nonumber
		\end{equation}
		This is an inequality on $q$, and one can show that it holds only for $q\leq 9$ by Magma \cite{Bosma}. Then for each prime power $q\leq 9$,  we  can compute the value $a:=(v-1,|X_{\alpha}|\cdot|\mathrm{Out}(X)|)$. Finally, we substitute the values $v,a$ back into the inequality $v\leq 18a^2$, and find that the inequality does not hold, leading to a contradiction.
	\end{example}
   For parabolic subgroups, we use similar processing method. 
   \begin{example}\label{ex2}
   	Suppose that $X=G_2(q)$ and $X_{\alpha}=[q^5]:\mathrm{GL}(2,q)$ is a parabolic subgroup, where $q=p^f$ for some prime $p$ and positive integer $f$. 	Note that $G_2(2)$ is not a simple group. Thus we assume that $q>2$. Then $v=|X|/|X_{\alpha}|=(q^6-1)/(q-1)$ and so $(v-1)_p=q$. By Lemma \ref{p-power}, there is a subdegree $d$ which is a power of $p$. Consequently, $(v-1,d)=q$. Furthermore, by Lemma \ref{s23}(iii) and Lemma \ref{s22}(v), we have $(q^6-1)/(q-1)=v\leq18\cdot \frac{r^2}{(r,\lambda)^2}\leq18q^2$   and so $q=2$, which is a contradiction. 
   \end{example}
   
   In \cite[Section 3.1]{Lu}, the authors  presented a computational method to search for positive integers $(v,b,r,k,\lambda)$ as potential parameters for the design $\mathcal{D}$. This method  is  convenient for calculating some small values of $v$, such as sporadic cases in Table \ref{large_maximal}. For completeness, we will repeat it in the following paragraph.
   
   Given the value of $v$, we use Magma \cite{Bosma} to search for positive integers $(v,b,r,k,\lambda)$ that satisfy the following conditions: 
	\begin{equation}\label{y}
		2\leq y\leq10,
	\end{equation}
	\begin{equation}\label{v}
		k<v<\frac{k^{2}-k}{y-1},
	\end{equation}
	\begin{equation}\label{r1}
		r(k-1)=\lambda(v-1),
	\end{equation}
	\begin{equation}\label{r2}
		(y-1)(r-1)=(k-1)(\lambda-1),
	\end{equation}
	\begin{equation}\label{divisible}
		y\mid k,~y\mid (r-\lambda),
	\end{equation}
	\begin{equation}\label{b}
		b=\frac{vr}{k}.\\
	\end{equation}
	We provide specific calculation steps:
	\begin{enumerate}
		\item [(1)] For a given $v$ and $y$ obtained from \eqref{y}, substitute $v$ and  $y$ into \eqref{v} to  obtain the range of positive integer $k$;
		\item [(2)] Substitute the above possible values of $k$ and $v$ into \eqref{r1} and \eqref{r2} to obtain all possible positive integer solutions for $(r,\lambda)$, where $r>k>\lambda>1$;
		\item [(3)] Check whether $(y,k,r,\lambda)$ satisfies the conditions in \eqref{divisible};
		\item [(4)] Substitute the positive integer $r$ and corresponding $v,k $ into \eqref{b} to check whether $b$ is a positive integer.
	\end{enumerate}
	Finally, we identify the parameters $(v,b,r,k,\lambda) $  that satisfy all above conditions  as potential  parameters for the design $\mathcal{D}$.

	\subsection{Maximal non-parabolic subgroups in Table \ref{large_maximal}}
	According to Lemma \ref{s23}(i), we have  $|H|^3>\lambda|G|$, which implies that $H$ is a large maximal subgroup of $G$. Therefore, by Lemma \ref{max}, the subgroup $H$ is either a parabolic subgroup or a subgroup such that $H\cap X=X_{\alpha}$ is listed in Table \ref{large_maximal}. We shall initially focus on the latter case.
	\begin{lemma}\label{non_parabolic}
		Assume Hypothesis \ref{hy_exceptional}. The pair $(X,X_{\alpha})$ does not belong to any of the cases  listed in Table \ref{large_maximal}, with the exception of the case $(G_2(q),A_2^{\pm}(q))$.
	\end{lemma}
	\begin{proof}
		For each of the candidate pairs $(X,X_{\alpha})$ without specific $q$ value conditions, we can calculate the parameters $v-1$ and $|X_{\alpha}|$. By  using the method
		in Section \ref{Comp}, the computation in Magma \cite{Bosma} gives three polynomials  with integer coefficients $f(q)$, $g(q)$ and $h(q)$ such that
		\[
		f(q)\cdot(v-1)+g(q)\cdot|X_{\alpha}|=h(q).
		\]
		According to Lemma \ref{s23}(ii) and Lemma \ref{stabilizer}, we have $\frac{r}{(r,\lambda)}\leq(v-1,|G_{\alpha}|)\leq h(q)\cdot |\mathrm{Out}(X)|$. It follows from Lemma \ref{s22}(v) that
		\begin{equation}\label{v_bound}
			v\leq2(y-1)\cdot\frac{r^2}{(r,\lambda)^2}\leq	18\cdot\left(h(q)\cdot |\mathrm{Out}(X)|\right)^2.
		\end{equation}
		This inequality \eqref{v_bound} is on $q$, and after computation,  we find that only a few small values of $q$ satisfy this inequality. Furthermore, for these $q$ values, and the specific remaining $q$ values listed in Table \ref{large_maximal}, we  directly verify whether inequality $v\leq18\cdot(v-1,|X_{\alpha}|\cdot|\mathrm{Out}(X)|)^2$ holds. We provide Example \ref{ex1} to demonstrate these calculation processes. Finally, for the  $q$ values that satisfy   above inequality,  we  search for potential parameters for $\mathcal{D}$. It is worth noting that  most  cases are excluded in the second step, and  no possible parameters are found  in the final calculation. The only remaining case is that $(G_2(q),A_2^{\pm}(q))$. In this case, we have $v=\frac{1}{2}q^3(q^3\pm1)$ and $h(q)=12(q^3\mp1)$. Then inequality \eqref{v_bound} becomes 
		$\frac{1}{2}q^3(q^3\pm1)\leq18\cdot\left(24f(q^3\mp1)\right)^2$ since $|\mathrm{Out}(X)|\leq 2f$, where $q=p^f$ for some prime $p$ and positive integer $f$. This is always true.
	\end{proof}

\begin{lemma}\label{A_2}
	Assume Hypothesis \ref{hy_exceptional}. If $X=G_2(q)$ with $q\geq3$, then  $X_{\alpha}$ is not a maximal subgroup of type $A_2^{\pm}(q)$.
\end{lemma}
\begin{proof}
	Suppose to the contrary that $X_{\alpha}$ is of type  $A_2^{\pm}(q)$. According to  \cite[Table 4.1]{Wilson2009}, we know that $X_{\alpha}=\mathrm{SL}(3,q):2$ or $\mathrm{SU}(3,q):2$. Note that $|X|=q^6(q^6-1)(q^2-1)$. Then $v=\frac{1}{2}q^3(q^3+\epsilon)$, where $\epsilon=\pm1$. Combining Lemma \ref{s23}(ii)(iii) with \cite[Section 4, Case 1 and Section 3, Case 8]{Saxl2002}, we conclude that $\frac{r}{(r,\lambda)}$ divides $\frac{1}{2}(q^3-1)$ for $q$ odd and  $\frac{r}{(r,\lambda)}$ divides $q^3-1$ for $q$ even if $\epsilon=+1$; $\frac{r}{(r,\lambda)}$ divides $\frac{1}{2}(q^3+1)$ for $q$ odd and  $\frac{r}{(r,\lambda)}$ divides $3(q^3+1)$ for $q$ even if $\epsilon=-1$. For example, suppose that $q$ is odd. From the factorization $\Omega(7,q)=G_2(q)N^{\epsilon}_1$, it follows that the suborbits of $\Omega(7,q)$ are unions of $G_2(q)$-suborbits. The $\Omega(7,q)$-subdegree are $q^6-1$, $\frac{1}{2}q^2(q^3-\epsilon)$ and $\frac{1}{2}(q-3)$ times $q^2(q^3-\epsilon)$. Since $\gcd(v-1,q)=1$,  by Lemma \ref{s23}(iii), we have $\frac{r}{(r,\lambda)}$ divides $\frac{1}{2}(q^3-\epsilon)$.
	
	We  give the details only for case where $\epsilon=+1$ and $q$ is odd, since  other cases are similar. Since $\frac{r}{(r,\lambda)}$ divides $\frac{1}{2}(q^3-1)$, we have $\frac{r}{\lambda}=\frac{r}{(r,\lambda)}/\frac{\lambda}{(r,\lambda)}=\frac{1}{2u}(q^3-1)$, where $u$ is a positive integer. Suppose that $u\geq3$.  Since $y\leq10$, it follows from Lemma \ref{s22}(v) that $v-1=\frac{1}{2}(q^6+q^3-2)<2(y-1)\cdot\frac{(q^3-1)^2}{4u^2}\leq\frac{1}{2}(q^6-2q^3+1)$, which is a contradiction. Thus we have $u=1$ or $2$.
	
 If $u=1$, then by Lemma \ref{s22}(v), we have $\frac{1}{4}(y-1)(q^3-1)^2<v-1=\frac{1}{2}(q^3+2)(q^3-1)<\frac{1}{2}(y-1)(q^3-1)^2$, so $y=3$. Recall that $k=\frac{\lambda(v-1)}{r}+1$. Then we deduce that $k=q^3+3$. Furthermore, by Lemma \ref{s22}(i), i.e., $(y-1)(r-1)=(k-1)(\lambda-1)$, we  compute that $\lambda=\frac{1}{3}q^3$ and $r=\frac{1}{6}q^3(q^3-1)$. Hence $b=\frac{vr}{k}=\frac{q^6(q^6-1)}{12(q^3+3)}$.  Since $b$ is a positive integer, it follows that $q^3+3$ divides $72$,  which is impossible.
	
 If $u=2$, then $k=2q^3+5$. Furthermore, we can similarly deduce that $6\leq y\leq 9$  by   Lemma \ref{s22}(v). We give details only for the case $y=6$, as other cases are similar. For $y=6$,  by Lemma \ref{s22}(i), we obtain that $\lambda=\frac{4(2q^3-1)}{3(q^3+7)}<3$, which contradicts the fact  that  $y< \lambda$ from Lemma \ref{s22}(ii).	This completes the proof.	
\end{proof}

	\subsection{Maximal parabolic subgroups}
	
	In this subsection, we deal with the case $H=G_{\alpha}$ is a maximal parabolic subgroup.
	\begin{lemma}\label{suzuki}
		Assume Hypothesis \ref{hy_exceptional}. If $X={ }^2B_2(q)$ with $q=2^{2n+1}\geq8$, then $X_{\alpha}$ is not  a maximal parabolic subgroup $[q^2]:(q-1)$.
	\end{lemma}
	\begin{proof}
		Since $|X|=q^2(q^2+1)(q-1)$, it follows from the point-transitivity of $X$ that $v=|X|/|X_{\alpha}|=q^2+1$. Then, by Lemma \ref{s23}(ii), we have $r\mid \lambda q^2$. Note that $r>\lambda$ and $q=2^{2n+1}$. Then we deduce that $r$ is even. Recall that $(y-1)(r-1)=(k-1)(\lambda-1)$ by Lemma \ref{s22}(i). Therefore, $(k-1)_2\leq(y-1)_2$. We write that  $(y-1)_2=a$, where $a=1,2,4$ or $8$ since $y\leq10$. 
		
		Suppose that $a=1$. Then we have $(k-1,q^2)=1$. It follows from Lemma \ref{s21}(i) that $\frac{r}{\lambda}=\frac{v-1}{k-1}=\frac{q^2}{k-1}$. Therefore, $r\geq q^2=v-1$, which implies that $r=v-1=q^2$ by Lemma \ref{s22}(ii). Moreover, we have $\lambda=k-1$ and so $(r,\lambda)=1$. From \cite[Theorem 1]{Zhang2020} (see also \cite[Theorem 1]{Alavi2020}), we know that $k=q$ and $\lambda=q-1$. It follows from Lemma \ref{s22}(iii) that $y\mid q$ and  $y\mid (q^2-q+1)$, which is impossible. Therefore, we can always assume that $(r,\lambda)>1$.
		
	Form now on, we suppose that $a\geq2$. Note that we have $(k-1,q^2)\mid a$. From $\frac{r}{\lambda}=\frac{q^2}{k-1}$,  there is an integer $m$ such that $r=\frac{q^2m}{a}$ and $\lambda=\frac{m(k-1)}{a}$. We claim that $m$ is even and $m<a$. Suppose to the contrary that $m$ is odd. Then $(k-1)_2=a=(y-1)_2$ and so $\lambda_2=\left(\frac{m(k-1)}{a}\right)_2=1$, i.e., $\lambda$ is odd. Thus we   conclude that $\left((k-1)(\lambda-1)\right)_2>\left((y-1)(r-1)\right)_2$, which contradicts   Lemma \ref{s22}(ii). If $m=a$, then  $r=q^2$ and $\lambda=k-1$. It follows from $(r,\lambda)>1$ that $k$ is odd.  Note that $b=\frac{vr}{k}=\frac{q^2(q^2+1)}{k}$ is an integer. Then we know that $k\mid (q^2+1)$. By solving the following system of equations:
		\begin{equation}\label{equa1}
			r(k-1)=\lambda(v-1),
			\nonumber
		\end{equation}
		\begin{equation}\label{equa2}
			(y-1)(r-1)=(k-1)(\lambda-1),\nonumber
		\end{equation}
		we get that $\lambda=\frac{(k-1)(k-y)}{(k-1)^2-(y-1)q^2}$. Since $\lambda=k-1$, we have $k-y=(k-1)^2-(y-1)q^2$, i.e., $k^2-3k+(y+1)-(y-1)q^2=0$. It follows from $k\mid (q^2+1)$ that $k$ divides $k^2-3k+(y+1)-(y-1)q^2+(y-1)q^2+(y-1)=k^2-3k+2y$. Then $k\mid 2y$ and so $k\leq y$, which contradicts Lemma \ref{s22}(ii). Therefore, since $m$ is even and $m<a$, we know that $r$ is a power of $2$ that is less than $q^2$. Moreover, from $\lambda=\frac{m(k-1)}{a}$, we know that $k$ is odd. Thus $k$ divides $q^2+1$.  Next, we only need to analyze the possible values of $m$ one by one  when $a$ equals $4$ or $8$.
		
		Suppose that $a=4$. Then $y=5$. Let $r=\frac{q^2m}{4}$ and $\lambda=\frac{m(k-1)}{4}$, where $m$ is  an even integer. Based on the previous proof, we   conclude that $m=2$. From   $\frac{k-1}{2}=\lambda=\frac{(k-1)(k-5)}{(k-1)^2-4q^2}$,  we deduce that $k^2-4k+11-4q^2=0$. Combining this with the fact that $k\mid (q^2+1)$, we   conclude that $k\mid 15$ and so $k=15$. It follows that $q^2=44$, which is impossible.	Suppose that $a=8$. Then $y=9$. Similarly, we can suppose that  $r=\frac{q^2m}{8}$ and $\lambda=\frac{m(k-1)}{8}$, where $m=2,4$ or $6$.  From $\frac{m(k-1)}{8}=\lambda=\frac{(k-1)(k-9)}{(k-1)^2-8q^2}$, we   compute that $mk^2-(2m+8)k+m+72-8mq^2=0$. It follows from $k$ dividing $q^2+1$ that $k\mid (72+9m)$. Thus we have $k\mid 45$, $k\mid 27$ and $k\mid 63$  for $m=2,4$ and $6$ respectively. Then, considering the conditions $y\mid k$ and $y\neq k$, we deduce that $k=45,27,63$ respectively. However,  we can directly verify that there are no prime powers $q$ that satisfy the above equations for these $k$ values. This completes the proof.
	\end{proof}
	
		\begin{lemma}\label{G2}
	Assume Hypothesis \ref{hy_exceptional}.	If $X={ }^2G_2(q)$ with $q=3^{2n+1}\geq27$, then $X_{\alpha}$ is not  a maximal parabolic subgroup $[q^3]:(q-1)$.
	\end{lemma}
	\begin{proof}
	Our argument here is similar to that in the proof of Lemma \ref{suzuki}. 	Since $|X|=q^3(q^3+1)(q-1)$, it follows from the point-transitivity of $X$ that $v=|X|/|X_{\alpha}|=q^3+1$.  Then by Lemma \ref{s23}(ii), we have $r\mid \lambda q^3$. Note that $r>\lambda$ and $q=3^{2n+1}$. Then we deduce that $r\equiv 0\pmod{3}$.
	Recall from Lemma \ref{s22}(i) that $(y-1)(r-1)=(k-1)(\lambda-1)$. Therefore, $(k-1)_3\leq(y-1)_3$. We write that  $(y-1)_3=a$, where $a=1,3$ or $9$ since $y\leq10$. 
	
	Suppose that $a=1$. Then we have $(k-1,q^3)=1$. It follows from Lemma \ref{s21}(i) that $\frac{r}{\lambda}=\frac{v-1}{k-1}=\frac{q^3}{k-1}$. Therefore, $r\geq q^3=v-1$ which implies that $r=v-1=q^3$ by Lemma \ref{s22}(ii). Moreover, we have  $\lambda=k-1$ and so $(r,\lambda)=1$. From \cite[Theorem 1]{Zhang2020} (see also \cite[Theorem 1]{Alavi2020}), we know that $k=q,\lambda=q-1$ or $k=q^2,\lambda=q^2-1$. It follows from $y\mid (r-\lambda)$ that $y\mid (q^3-q+1)$ or  $y\mid (q^3-q^2+1)$ by Lemma \ref{s22}(ii), which contradicts  that $y\mid k$. Therefore, we can always assume that $(r,\lambda)>1$ and so $(k,q^3)=1$.
	
	From now on, we suppose that $a\geq3$. Let $r=\frac{q^3m}{a}$ and $\lambda=\frac{m(k-1)}{a}$ for some integer $m$. Suppose that $m=a$. Then $r=q^3$ and $\lambda=k-1$. Similar to the proof of Lemma \ref{suzuki}, we can deduce that $k\mid 2y$. Note that $k\equiv1\pmod{3}$ since  $(r,\lambda)>1$. Then  we get $k=y$ since $y\equiv1\pmod{3}$, which contradicts  Lemma \ref{s22}(ii).
	
	Suppose that $a=3$. Then $y=4$ or $7$. Let $r=\frac{q^3m}{3}$ and $\lambda=\frac{m(k-1)}{3}$, where $m=1$ or $2$. We  give the details only for the case where $y=4$ and $m=1$, as other cases follow similarly. From $\frac{k-1}{3}=\lambda=\frac{(k-1)(k-4)}{(k-1)^2-3q^3}$,  we   deduce that $k^2-5k+13-3q^3=0$. Note that $b=\frac{vr}{k}=\frac{q^3(q^3+1)}{k}$ is an integer. Then $k\mid (q^3+1)$, and so we conclude that $k\mid 16$.  Recall that $y=4$ divides $k$ and $\lambda=\frac{k-1}{3}$ is an integer. Thus we must have $k=16$. It follows that $q^3=16^2-5\cdot16+13=189$, which is impossible. Suppose that $a=9$. Then $y=10$. Similarly, we can suppose that  $r=\frac{q^2m}{9}$ and $\lambda=\frac{m(k-1)}{9}$, where $m\le8$.  From $\frac{m(k-1)}{9}=\lambda=\frac{(k-1)(k-10)}{(k-1)^2-9q^3}$, we   compute that $mk^2-(2m+9)k+m+90-9mq^3=0$. Note that $b=\frac{vr}{k}=\frac{mq^3(q^3+1)}{k}$ is an integer. Then $k\mid m(q^3+1)$, and so we   conclude that $k\mid (90+10m)$. It can be verified that for the cases $m\leq8$, there are no integers $k$ that satisfy all the above conditions simultaneously. For example, if $m=1$, then $k\mid 100$. By the fact that $y=10$ divides $k$ and $\lambda=\frac{k-1}{9}$ is an integer, we have $k=100$. It follows that $9q^3=100^2-11\cdot100+91=8991$, which is impossible.  This completes the proof.
	\end{proof}
	
	\begin{lemma}\label{parabolic}
		Assume Hypothesis \ref{hy_exceptional}. The point stabilizer $H=G_{\alpha}$ is not a parabolic subgroup of $G$.
	\end{lemma}
	\begin{proof}
		Firstly, by Lemmas \ref{suzuki} and \ref{G2},  we know that $(X, X_{\alpha})\neq\left({ }^2B_2(q),[q^2]:(q-1)\right)$ or  $\left({ }^2G_2(q),[q^3]:(q-1)\right)$. We then assume that $X\neq E_6(q)$. By Lemma \ref{p-power}, there is a unique subdegree which is a power of $p$. Moreover, we have $(v-1)_p\leq 2q$  for all parabolic subgroups, with the equality holding only when $q=2^f$. Then by Lemma \ref{s23}(iii) and Lemma \ref{s22}(v), we have
		\begin{equation}
			v\leq2(y-1)\cdot \frac{r^2}{(r,\lambda)^2}\leq18(v-1,d)^2\leq 18\cdot 4q^2.
		\end{equation} 
		Similar to the final computation process in Lemma \ref{non_parabolic}, we can eliminate this case. Specifically, one can refer to the calculation in Example \ref{ex2} for details.
		
		Finally, we suppose that $X=E_6(q)$. If $G$ contains a graph automorphism
     	and $H=P_i$ with $i=2$ or $4$, then there is a unique subdegree which is a
		power of $p$ (see \cite[p.345]{Saxl2002}). Therefore, we can deal with these two cases in the same way as the proof in the previous paragraph. If $H\cap X$ is $P_1$ with type $D_5(q)$, then $v=(q^8+q^4+1)(q^9-1)/(q^8-1)$. Note that the right coset action of $G$ on $H$ is rank $3$ by \cite{Liebeck1986}. Two non-trivial subdegrees are $d_1:=q(q^8-1)(q^3+1)/(q-1)$ and $d_2:=q^8(q^5-1)(q^4+1)/(q-1)$, respectively. By Lemma \ref{s22}(iii) and Lemma \ref{s22}(v), we have $v\leq18\cdot (v-1,d_1)^2$, which is impossible by Magma \cite{Bosma}. If $H\cap X$ is $P_3$ with type $A_1(q)A_4(q)$, then $v=(q^3+1)(q^4+1)(q^9-1)(q^{12}-1)/(q-1)(q^2-1)$. This case can be ruled out similar to the proof of Example \ref{ex1} and we omit the details. This completes the proof.
		\end{proof}
	\noindent\textbf{Proof of Theorem \ref{main}.}  By Lemma \ref{s23}(i),  the point stabilizer $G_{\alpha}$ is maximal in $G$. Thus we can apply Lemma \ref{max} and analyze each possible case. In Lemmas \ref{non_parabolic} and \ref{A_2}, we excluded the case where $G_{\alpha}$ is a maximal non-parabolic subgroup of $G$. In Lemmas \ref{suzuki}, \ref{G2} and \ref{parabolic}, we handled the case where $G_{\alpha}$ is a parabolic subgroup of $G$. Putting together,  this completes the proof of Theorem \ref{main}.

	\vspace*{10pt}

	\begin{center}
		\scriptsize
		\setlength{\bibsep}{0.5ex}  
		\linespread{0.5}
		\bibliographystyle{plain}

	\end{center}
	
\end{document}